\documentclass[12pt,a4paper]{article}
\usepackage[T2A]{fontenc}
\usepackage[cp1251]{inputenc}
\usepackage[russian,english]{babel}
\usepackage{amssymb}
\usepackage{amsfonts}
\usepackage[tbtags]{amsmath}
\usepackage{amscd}
\usepackage{amsthm}
\usepackage{mathtext}
\usepackage{pifont}
\usepackage{bm}
\usepackage{array}
\usepackage{dcolumn}
\usepackage{hhline}
\usepackage{multirow}
\usepackage{graphicx}
\usepackage{rotating}
\usepackage{calc}
\usepackage{tabularx}
\usepackage{afterpage}
\usepackage{ifthen}
\usepackage{caption2}
\usepackage{substr}
\usepackage[mathscr]{eucal} 
\usepackage{mathrsfs}       
\usepackage{hypbmsec}      
\usepackage{latexsym}       
\usepackage{xypic}          
\RequirePackage{soul}
\RequirePackage{verbatim}    
\RequirePackage{chapterbib}
\RequirePackage{enumerate}

\newtheorem{theorem}{Theorem}
\newtheorem{corollary}{Corollary}
\newtheorem{lemma}{Lemma}

\usepackage[left=2.20cm,right=2.20cm,top=1.50cm,bottom=1.50cm]{geometry}

\begin{document}

\begin{center}
	\textbf{{\LARGE Non-alternating Hamiltonian Lie algebras in three variables}}
\end{center}

\begin{center}
	\textbf{A.~V.~Kondrateva}
\end{center}

\begin{center}
	\textit{National Research Lobachevsky State University of Nizhny Novgorod, Gagarin ave. 23, Nizhny Novgorod, 603950 Russia}
	
	E-mail: alisakondr@mail.ru
\end{center}

\begin{flushleft}
	{\small \textbf{Abstract}. Non-alternating Hamiltonian Lie algebras in three variables over a perfect field of characteristic 2 are considered. A classification of non-alternating Hamiltonian forms over an algebra of divided powers in three variables and of the corresponding simple Lie algebras is given. In particular, it is shown that the non-alternating Hamiltonian form may not be equivalent to its linear part. It is proved that the Lie algebras which correspond to the nonequivalent forms are not isomorphic.}
\end{flushleft}

\begin{center}
	{\small Keywords ands phrases: \textit{Non-alternating Hamiltonian Lie algebra, non-alternating Hamiltonian form, characteristic two}}
\end{center}

\section{Introduction}

Non-alternating Hamiltonian algebras over a field of characteristic 2 were first constructed in 1993 by Lin Lei \cite{N} as Lie algebras of polynomials in divided powers with the symmetric Poisson bracket $\{f, g\} = \sum\limits_i \partial_i f \partial_i g$. In the case when the heights of the variables are equal to 1, non-alternating Hamiltonian Lie algebras are isomorphic to the first series of simple Lie algebras built by I. Kaplansky \cite{Kap}. In \cite{bgl, L}, symmetric differential forms in divided powers are introduced and non-alternating Hamiltonian Lie algebras similar to Hamiltonian Lie superalgebras of characteristic zero with respect to Poisson brackets with constant coefficients are studied. 
The classification of alternating Hamiltonian forms over an algebra of truncated polynomial is obtained by M.~I.~Kuznetsov, S~.A.~Kirillov (\cite{KK}). The complete classification over a divided powers algebra is build by S.~M.~Skryabin (\cite{S, SAr}). In \cite{SS, Sk} the general theory of alternating Hamiltonian Lie algebras in divided powers is developed. The non-alternating case is considered in \cite{dep}.
In particular, the classification of non-alternating Hamiltonian forms with constant coefficients over the divided powers algebra is obtained. It is proved that any filtered deformation of a graded non-alternating Hamiltonian Lie algebra is determined by a non-alternating Hamiltonian form with polynomial coefficients. One of the main results of \cite{dep} is the proof of the equivalence of the non-alternating Hamiltonian form $\omega$ with polynomial coefficients to its initial form $\omega(0)$, provided that the canonical form of $\omega(0)$ contains $(dx_i)^{(2)}$ or $dx_idx_j + (dx_j)^{(2)}$ for some variable $x_i$ of the height greater than 1. The problem of the classification of non-alternating Hamiltonian forms for which this condition is not satisfied remains open. In this paper, a classification of non-alternating Hamiltonian forms over algebras of divided powers in three variables with polynomial coefficients is obtained, and it is shown that if the above condition for the heights of the variables is not fulfilled, the non-alternating Hamiltonian form $\omega$ may not be equivalent to $\omega(0)$. In addition, the dimensions of the corresponding simple non-alternating Hamiltonian Lie algebras are found and it is proved that the Lie algebras which correspond to the nonequivalent forms are not isomorphic. \\ [0.5 cm]

ACKNOWLEDGEMENTS. The investigation is funded by RFBR according to research project № 18-01-00900 and Ministry of Science and Higher education of Russian Federation (project 0729-2020-0055). 

\section{Preliminaries}

Here we recall some basic definitions and results from \cite{dep}. Let $K$ be a perfect field of characteristic 2, $E$ be a vector space over $K$ and $\mathscr{F}\colon E=E_{0}\supseteq E_{1}\supseteq \ldots \supseteq E_{r}\supset E_{r+1}=\{0\}$ be a flag of $E$. A basis $\{ x_{1}, \ldots, x_{n} \}$ of $E$ is called coordinated with the flag $\mathscr{F}$ if $\{ x_{1}, \ldots, x_{n} \}\cap E_{i}$ is a basis of $E_{i}$, $i=1, \ldots, r$.

We use the notations and the definitions of algebras $O(E)$, $O(\mathscr{F})$, $O(n, \overline{m})$,  $W(E)$, $W(\mathscr{F})$, 
$W(n, \overline{m})$ from \cite{KSh}. The maximal ideal of $O(\mathscr{F})$ is denoted by $\mathfrak{m}$. Let $\{ x_{1}, \ldots, x_{n} \}$ be a basis of $E$ coordinated with the flag $\mathscr{F}$, $m_{i}=\min\{j ~|~ x_{i}\notin E_{j}\}$ be the height of $x_{i}$. The algebra $O(\mathscr{F})$ is isomorphic to the algebra $O(n, \overline{m})$, $\overline{m}=(m_{1}, \ldots, m_{n})$. The flag $\mathscr{F}$ is restored by the $n$-tuple $\overline{m}$, $E_{j} = \langle x_{i} ~|~ m_{i}>j \rangle$, in particular $E_{1} = \langle x_{i} ~|~ m_{i}>1 \rangle$.

A flag $\mathscr{F}\colon E=E_{0}= E_{1}= \ldots = E_{m-1}\supset E_{m}=\{0\}$ is called almost trivial. If $m=1$, then $\mathscr{F}$ is trivial.

Let $S\Omega(\mathscr{F}) = \bigoplus_{i\geqslant 0} S\Omega^{i}(\mathscr{F})$ be the algebra of symmetric differential forms, which is an $O(\mathscr{F})$-algebra of the divided powers (see also \cite{bgl, KKCh}).

Recall the definition of the differential $d\colon S\Omega^{r}(\mathscr{F}) \rightarrow S\Omega^{r+1}(\mathscr{F})$.
\begin{multline*}
(d\omega)(D_{1},\ldots,D_{r+1})=\sum_{i=1}^{r+1}D_{i}(\omega(D_{1},\ldots, \widehat{D_{i}}, \ldots,D_{r+1}))+ \\
+\sum_{i<j}\omega([D_{i},D_{j}], D_{1},\ldots, \widehat{D_{i}},\ldots, \widehat{D_{j}}, \ldots,D_{r+1}),
\end{multline*}	\\ [-0.7 cm]
$$d\omega^{(r)}= \omega^{(r-1)}d\omega, \quad df(D)= D(f), ~~f\in O(\mathscr{F}).$$

An isomorphism $\sigma\colon O(\mathscr{F})\rightarrow O(\mathscr{F}')$ is called admissible if for $r\geqslant 0$, $f \in \mathfrak{m}$ the inclusions $f^{(r)}\in O(\mathscr{F})$, $(\sigma f)^{(r)}\in O(\mathscr{F}')$ are both either fulfilled or not and if the inclusions are both fulfilled, then $(\sigma f)^{(r)}=\sigma (f^{(r)})$. This is equivalent to the fact that the mapping $D\mapsto \sigma\circ D\circ \sigma^{-1}$ defines an isomorphism $W(\mathscr{F})\rightarrow W(\mathscr{F}')$. An admissible isomorphism $\sigma$ extends to the isomorphism $S\Omega(\mathscr{F})\rightarrow S\Omega(\mathscr{F}')$ by the rule
$$(\sigma\omega)(D_{1},\ldots, D_{k})= \sigma(\omega(\sigma^{-1}D_{1}\sigma, \ldots, \sigma^{-1}D_{k}\sigma)).$$
Moreover, $\sigma d\omega = d\sigma\omega$, $\sigma(\omega_{1}\omega_{2}) = \sigma\omega_{1}\cdot \sigma\omega_{2}$ and $\sigma(\omega^{(k)})= (\sigma\omega)^{(k)}$.

Two differential forms $\omega$ and $\omega'$ will be called \textit{equivalent} if there is an admissible isomorphism $\sigma$ such that $\sigma\omega = \omega'$.
A symmetric $2$-form $\omega\in S\Omega^{2}(\mathscr{F})$ is called \textit{non-alternating} if there exists $D\in W(\mathscr{F})$ such that $\omega(D,D)\neq 0$.
Let $\{x_i\}$ be a basis of $E$  coordinated with the flag $\mathscr{F}$ and $\omega = \sum \omega_{ii}(dx_i)^{(2)} + \sum\limits_{i<j} \omega_{ij}dx_idx_j$. Set $\omega_{ji}= \omega_{ij}$ for $i<j$. A form $\omega$ is called \textit{nondegenerate} if $\det (\omega_{ij})$ is invertible in $O(\mathscr{F})$. 

A closed nondegenerate non-alternating form $\omega\in S\Omega^{2}(\mathscr{F})$ is called a \textit{non-alternating Hamiltonian} one. The corresponding Lie algebra of Hamiltonian vector fields is denoted by $\widetilde{P}(\mathscr{F}, \omega)$,
$$\widetilde{P}(\mathscr{F}, \omega)= \{ D\in W(\mathscr{F}) ~|~ D\omega=0 \}.$$

Moreover,
$D\in \widetilde{P}(\mathscr{F}, \omega)$ is equal to $D_{f}=\sum\limits_{i,j=1}^{n}\overline{\omega}_{ij} \partial_{j}f\partial_{i}$ for some $f \in \widetilde{O}(\mathscr{F}) = O(\mathscr{F}) + \langle x_{1}^{(2^{m_{1}})}, x_{2}^{(2^{m_{2}})},\ldots x_{n}^{(2^{m_{n}})} \rangle$.
Here $(\overline{\omega}_{ij}) = (\omega_{ij})^{-1}$ (see \cite{dep}).

The mapping $f\mapsto D_{f}$ is an isomorphism of the Lie algebra $\widetilde{P}(\mathscr{F}, \omega)$ and the Lie algebra $\widetilde{O}(\mathscr{F})\big/ K$ with the Poisson bracket
\begin{eqnarray}\label{eqeq 0.7}
\{ f, g \} = D_{f}(g) = D_{g}(f)= \sum_{i,j=1}^{n}\overline{\omega}_{ij} \partial_{i}f\partial_{j}g. 
\end{eqnarray}

Sometimes we will identify $\widetilde{P}(\mathscr{F}, \omega)$ with the Lie algebra $(\widetilde{O}(\mathscr{F})\big/ K, ~\{~ ,~\})$.     
Let $$P(\mathscr{F}, \omega)=\{ D_{f} ~|~ f\in O(\mathscr{F})\big/ K \}.$$

It follows from (\ref{eqeq 0.7}) that $P(\mathscr{F}, \omega)$ is an ideal of $\widetilde{P}(\mathscr{F}, \omega)$ of codimension $n$. A Lie algebra $\mathscr{L}$, $P^{(1)}(\mathscr{F}, \omega)\subseteq \mathscr{L} \subseteq \widetilde{P}(\mathscr{F}, \omega)$ will be called a \textit{non-alternating Hamiltonian Lie algebra}. 

We will say, that $\mathscr{L}$ is associated with a pair $(\mathscr{F}, \omega)$. If the forms $\omega$, $\omega'$ are equivalent and $\sigma\colon O(\mathscr{F})\rightarrow O(\mathscr{F}')$ is an admissible isomorphism such that $\sigma(\omega)= \omega'$, then Lie algebras $P(\mathscr{F}, \omega)$, $P(\mathscr{F}', \omega')$ are isomorphic, $\sigma(P(\mathscr{F}, \omega))= P(\mathscr{F}', \omega')$, where $\sigma(D_{f})= \sigma\circ D_{f} \circ\sigma^{-1} =D_{\sigma(f)}$.

Put $W=W(\mathscr{F})$, $V=W\big/ \mathfrak{m}W$. The form $\omega(0)$ is a nondegenerate symmetric bilinear form on $V$: for $D_{1}, D_{2}\in W$ ~$\omega(0)(\overline{D_{1}}, \overline{D_{2}}) = \omega(D_{1}, D_{2})(0) = \overline{\omega(D_{1}, D_{2})}\in O(\mathscr{F})\big/ \mathfrak{m}\cong K$. Employing the natural pairing $V$ and $E\cong \mathfrak{m}\big/ \mathfrak{m}^{(2)}$, $\langle \overline{D}, \overline{f} \rangle = \overline{D(f)} \in O(\mathscr{F})\big/ \mathfrak{m}$, one obtains the isomorphism $\iota\colon V\rightarrow E\cong V^{\ast}$, $\iota(\overline{D}) = \omega(0)(\overline{D}, -)$. The dual form $\overline{\omega(0)}$ on $E$, $\overline{\omega(0)}(\iota(\overline{D_{1}}), \iota(\overline{D_{2}})) = \omega(0)(\overline{D_{1}}, \overline{D_{2}})$, has the matrix $(\omega_{ij}(0))^{-1}$ with respect to a basis $\{x_{1}, \ldots, x_{n}\}$ dual to a basis $\{\overline{\partial_{1}}, \ldots, \overline{\partial_{n}}\}$ of $V$, were $(\omega_{ij}(0))$ is the matrix of $\omega(0)$ with respect to a basis $\{\overline{\partial_{1}}, \ldots, \overline{\partial_{n}}\}$. The space of null-vectors of $E$ with respect to the form $\overline{\omega(0)}$ is denoted by $E^{0}$. Let $\{x_{1}, \ldots, x_{n}\}$ be a basis of $E$ coordinated with the flag $\mathscr{F}$. Since $W= \langle \overline{\partial_{1}}, \ldots, \overline{\partial_{n}} \rangle \oplus \mathfrak{m}W$, we identify $V$ with $\langle \partial_{1}, \ldots, \partial_{n} \rangle$.

Later on, we will consider the dual flag of $V$, $0=V_{0}\subseteq V_{1}\subseteq \ldots \subset V_r=V$, where $V_i= \text{Ann}\,E_i$. To avoid the abuse of notations we will denote it by $\mathscr{F}$ as well.

The set $\{ i ~|~ \overline{\omega}_{ii}\neq 0 \}$ is denoted by $I$. E.g., if $\omega= dx_{1}dx_{2}+ (dx_{3})^{(2)}$, then $I= \{3\}$, if $\omega= dx_{1}dx_{2}+ (dx_{2})^{(2)}+ (dx_{3})^{(2)}$, then $I= \{1,3\}$. Note that  $\overline{\omega}_{ii}\in K$. Indeed, since $0=[D_{x_i}, D_{x_i}]=D_{\{x_i,x_i\}}$ and $D_f=0$ if and only if $f\in K$, we have $\overline{\omega}_{ii}=\{x_i,x_i\}\in K$.

We will use the following theorems from \cite{dep}. Here $G'$ is the subgroup of admissible automorphisms $\varphi$ of $O(\mathscr{F})$ such that $\varphi(x) \equiv x$ (mod $\mathfrak{m}^2$), $x\in E$, $dz_{r} =x_{r}^{(2^{m_{r}}-1)}dx_{r}$, $\widetilde{\mathfrak{m}}^{(j)}S\Omega^{1} = \langle T \rangle$, were $T= \{x^{(\alpha)}dx_{k}, ~|\alpha|\geqslant j, ~k=1,\ldots,n\}\diagdown \{x_{r}^{(2^{l})}x_{s}^{(2^{t})}dx_{i}, x_{q}^{(2)}dx_{i}, ~i,q\in I\}$.

\begin{theorem}[\cite{dep}, Theorem~3.1] \label{thth 3.1}
	Let $\omega$ be a non-alternating Hamiltonian form, $\omega= \omega(0) + d\psi + \eta$, where $\eta\in \langle dz_{r}dz_{s} \rangle$, $r,s=1,\ldots,n$, $\psi\in \mathfrak{m}^{(j+1)}S\Omega^{1}$, $j\geqslant 1$. There exists an automorphism $\sigma \in G'$ such that $\sigma\omega = \omega(0) + d\widetilde{\psi} + \widetilde{\eta}$, where $\widetilde{\psi} \in \widetilde{\mathfrak{m}}^{(j+1)}S\Omega^{1}$, $\widetilde{\eta}\in \langle dz_{r}dz_{s} \rangle$, $r,s=1,\ldots,n$.
\end{theorem}

\begin{theorem}[\cite{dep}, Theorem~3.2] \label{thth 3.2}
	Let $\omega$ be a non-alternating Hamiltonian form, such that
	$$\omega = \omega(0) + d\varphi + \sum\limits_{i<j} b_{ij}dz_{i}dz_{j},$$ 
	where $\varphi \in \widetilde{\mathfrak{m}}^{(2)}S\Omega^{1}$ and $b_{ij}\in K$. Then 
	
	(1) if there is $i \in I$ such that $m_{i}>1$ or $m_{i}=1$ for any $i \in I$ and $b_{sj}=0$, $s \in I$, then $\omega$ is conjugated with respect to $G'$ to the form $\omega(0)$.
	
	(2) $\omega$ is conjugated with respect to $G'$ to the form $\omega(0) + \sum\limits_{i \in I}\sum\limits_{j \notin I} b_{ij}dz_{i}dz_{j} + \sum\limits_{i<j \in I} b_{ij}dz_{i}dz_{j}$.
	
	Here, we set $b_{ji}= b_{ij}$ in the case $i<j$.
\end{theorem}


The theorem below is basically the weak form of \cite[Th.~5.2]{dep}. Here we use the concept of the minimal flag of a transitive Lie algebra from \cite{K}.

\begin{theorem} \label{thth 5.2}
	(i) Let $\mathscr{L}$ be a non-alternating Hamiltonian Lie algebra associated with a pair $(\mathscr{F}, \omega)$, $\mathscr{L}_{0}$ be the standard maximal subalgebra of $\mathscr{L}$. The minimal flag $\mathscr{F}(\mathscr{L}, \mathscr{L}_{0})$ coincides with $\mathscr{F}$.
	
	(ii) Let $\psi\colon (\mathscr{L}, \mathscr{L}_{0}) \rightarrow (\mathscr{L}', \mathscr{L}'_{0})$ be an isomorphism of non-alternating Hamiltonian transitive Lie algebras associated with pairs $(\mathscr{F}, \omega)$, $(\mathscr{F}', \omega')$, $\mathscr{L}_{0}$, $\mathscr{L}'_{0}$ be the standard maximal subalgebras of $\mathscr{L}$, $\mathscr{L}'$, respectively. Suppose that $n= \dim E>2$. Then $\psi$ is induced by the admissible isomorphism $\varphi\colon O(\mathscr{F})\rightarrow O(\mathscr{F}')$ and $\varphi(\omega)=a\omega'$, $a\in K\backslash\{0\}$.
	
	If $\mathscr{L}$, $\mathscr{L}'$ are graded Lie algebras with the standard gradings then $\overline{\psi} = \text{gr}\,\psi \colon (\mathscr{L}, \mathscr{L}_{0}) \rightarrow (\mathscr{L}', \mathscr{L}'_{0})$ is a homogeneous isomorphism of transitive Lie algebras which is induced by the admissible linear isomorphism $\overline{\varphi} = \text{gr}\,\varphi \colon O(\mathscr{F})\rightarrow O(\mathscr{F}')$ and $\overline{\varphi}(\omega)=a\omega'$, $a\in K\backslash\{0\}$.
\end{theorem}
\begin{proof}
	(i) Let $L= \text{gr}\mathscr{L} = L_{-1}+ L_{0}+ \ldots$ be the associated graded Lie algebra, $L_{(0)}= \text{gr}\mathscr{L}_{0}$. If $G= P(\mathscr{F}, \omega)$ then gr$G= P(\mathscr{F}, \omega(0))$. Since $(\text{gr}G)^{(1)}\subseteq \text{gr}(G^{(1)})$, $P(\mathscr{F}, \omega(0))^{(1)}\subseteq \text{gr}(P(\mathscr{F}, \omega)^{(1)})$. Analogously, for $G= \widetilde{P}(\mathscr{F}, \omega)$,  gr$G= \widetilde{P}(\mathscr{F}, \omega(0))$. Consequently,
	$$P(\mathscr{F}, \omega(0))^{(1)}\subseteq L \subseteq \widetilde{P}(\mathscr{F}, \omega(0)),$$
	that is $L$ is one of non-alternating Hamiltonian Lie algebras associated with the pair $(\mathscr{F}, \omega(0))$. By \cite[Prop.~0.4]{K} the minimal flag $\mathscr{F}(L, L_{(0)})$ is defined by the $p$-nilpotency degrees of elements ad$_{L}l$, $l\in L_{-1}= W_{-1}$, where $W(\mathscr{F})= W_{-1} + W_{0}+ \ldots$ is the standard grading of $W(\mathscr{F})$. For $\partial_{i}\in W_{-1}$, ad$_{L}\partial_{i}$ acts on $\widetilde{O}(\mathscr{F})\big/ K$ as $\partial_{i}$. Hence, $\mathscr{F}(L, L_{(0)}) = \mathscr{F}$. According to \cite[Prop.~0.3]{K}
	$$\mathscr{F}= \mathscr{F}(L, L_{(0)})\leqslant \mathscr{F}(\mathscr{L}, \mathscr{L}_{0})\leqslant \mathscr{F}.$$
	Thus, $\mathscr{F}(\mathscr{L}, \mathscr{L}_{0})= \mathscr{F}$.
	
	(ii) The proof is based on the theory of truncated coinduced modules (see \cite{K,SS}). Here we mainly concentrate on some key-points of the proof.
	
	According to (i) the natural embedding $\tau\colon (\mathscr{L}, \mathscr{L}_{0}) \rightarrow (W(\mathscr{F}), W(\mathscr{F})_{0})$ is a minimal one. By the embedding theorem \cite[Th.~1.1]{K} $\psi$ is induced by an admissible isomorphism  $\varphi\colon O(\mathscr{F})\rightarrow O(\mathscr{F}')$, $\varphi(l)=\varphi l \varphi^{-1}$, $l\in \mathscr{L}$. Let $(\mathscr{L}, \mathscr{L}_{0})$ be a transitive subalgebra of $W(\mathscr{F})$, $Q$ be an $\mathscr{L}$-module which is at the same time a finitely generated $O(\mathscr{F})$-module. We say that $\mathscr{L}$ acts on $Q$ by derivations if
	$$l(fq)= l(f)q + f(lq), \quad l\in \mathscr{L}, \quad f\in O(\mathscr{F}), \quad q\in Q.$$
	It should be pointed out that as it follows from transitivity of $(\mathscr{L}, \mathscr{L}_{0})$ $Q$ is a free $O(\mathscr{F})$-module. As well as in \cite{SS} we will call $Q$ an $(O(\mathscr{F}), \mathscr{L})$-module. Denote $W(\mathscr{F})$ by $W$, $O(\mathscr{F})$ by $O$.
	
	Let $\mathscr{L}^{\sharp} =$ Prim$U(\mathscr{L})$ be the Lie algebra of the primitive elements of Hopf algebra $U(\mathscr{L})$. $\mathscr{L}^{\sharp}$ is the restricted Lie algebra with respect to the $p$-mapping in $U(\mathscr{L})$. Let $\mathscr{M}$ be the normalizer of $\mathscr{L}_{0}$ in $\mathscr{L}^{\sharp}$. Obviously, $\mathscr{M}$ is a  restricted subalgebra of $\mathscr{L}^{\sharp}$. According to the theory of truncated coinduced modules the category of $(O, \mathscr{L})$-modules is equivalent to the category of finite-dimensional restricted $\mathscr{M}$-modules. The equivalence is given by the $\mathscr{M}$-morphisms $\lambda_{Q}\colon Q\rightarrow Q\big/ \mathfrak{m}Q$. From the universal property of coinduced modules it follows that $\lambda_{Q}$ results in the isomorphism of the spaces of invariants, $Q^{\mathscr{L}}\cong (Q\big/ \mathfrak{m}Q)^{\mathscr{M}}$.
	
	Now we consider the $(O, \mathscr{L})$-module $Q= S\Omega^{2}(\mathscr{F})=$ Hom$_{O}(S_{O}^{2}W, O)$. Let $L= \text{gr}\mathscr{L} = L_{-1}+ L_{0}+ \ldots$. Then we have
	$$Q\big/ \mathfrak{m}Q\cong \text{Hom}_{K}(S^{2}(W\big/ W_{(0)}), O\big/ \mathfrak{m})\cong \text{Hom}_{K}(S^{2}L_{-1}, K)$$
	and $\lambda_{Q}(\omega) = \omega(0)$. Obviously,
	$$(Q\big/ \mathfrak{m}Q)^{\mathscr{M}}\cong \text{Hom}_{K}(S^{2}(W\big/ W_{(0)}), K)^{\mathscr{M}}\subseteq \text{Hom}_{K}(S^{2}(W\big/ W_{(0)}), K)^{\mathscr{L}_{0}}\cong \text{Hom}_{K}(S^{2}L_{-1}, K)^{L_{0}}.$$
	Since $L_{-1}$ is an absolutely irreducible $L_{0}$-module when $n>2$, by Schur's lemma Hom$_{K}(S^{2}L_{-1}, K)^{L_{0}} = \langle \omega(0) \rangle_{K}$ is a one-dimensional space. Since $\omega\in Q^{\mathscr{L}}$,
	$$1\leqslant \dim(Q\big/ \mathfrak{m}Q)^{\mathscr{M}}\leqslant \dim\text{Hom}_{K}(S^{2}L_{-1}, K)^{L_{0}}=1.$$
	Hence, $\dim Q^{\mathscr{L}}=1$ which results in $Q^{\mathscr{L}}= \langle \omega \rangle$. Now,
	$$\mathscr{L}'\varphi(\omega) = \psi(\mathscr{L})\varphi(\omega)= \varphi \mathscr{L} \varphi^{-1}\varphi(\omega) = 0.$$
	Consequently, $\varphi(\omega) = a\omega'$, $a\in K\backslash \{ 0 \}$.
	
	Suppose that $\mathscr{L}$, $\mathscr{L}'$ are graded Lie algebras with the standard gradings. Then $\overline{\psi} = \text{gr}\,\psi \colon \mathscr{L} \rightarrow \mathscr{L}'$ is an isomorphism of graded Lie algebras. Obviously, $\overline{\psi}$ is induced by the linear admissible isomorphism $\overline{\varphi} = \text{gr}\,\varphi \colon O(\mathscr{F})\rightarrow O(\mathscr{F}')$ and $\overline{\varphi}(\omega)=a\omega'$, $a\in K\backslash \{ 0 \}$.
\end{proof}

\textbf{Remark.} The proof of the assertion in (ii) is true for $n=2$ under additional restrictions. All we need is that $L_{-1}$ should be an absolutely irreducible $L_{0}$-module. For $\omega= (dx_{1})^{(2)}+ (dx_{2})^{(2)}$ this is true if the set of heights $\overline{m}= (m_{1}, m_{2})$ is not equal to $(1,1)$.  For $\omega= dx_{1}dx_{2}+ (dx_{2})^{(2)}$ $L_{0}$ acts irreducibly on $L_{-1}$ if $m_{1}> 1$ and $m_{2}> 1$.

Everywhere in what follows we will assume that $\dim E=3$.

\section{Bilinear non-alternating symmetric forms in three variables}

Let $V$ be a three-dimensional vector space over a perfect field $K$ of characteristic 2, $b$ be a bilinear nondegenerate non-alternating symmetric form on $V$, $V^{0}$ be the hyperplane consisting of isotropic vectors (i.e., $V^{0}=\{v\in V ~|~ b(v,v)=0\}$),
$\mathscr{F}\colon 0=V_{0}\subseteq V_{1}\subseteq \ldots \subset V$ be a flag of $V$ such that $V_{q}=V$ for sufficiently large $q$. 
Let $\{ e_{i} \}$ be a basis of $V$ coordinated with the flag $\mathscr{F}$, $m_{i}= \min\{ j ~|~ e_{i}\in V_{j}\}$ be the height of $e_{i}$. The orthogonal complement of a subspace $L\subseteq V$ with respect to $b$ is denoted by $L^{\perp}$. 

Pairs $(\mathscr{F}, b)$ and $(\mathscr{F}', b')$ are called \textit{equivalent} if there is a linear automorphism $\varphi\colon V\rightarrow V$ such that $\varphi(\mathscr{F}) = \mathscr{F}'$ and $b(u,v) =b'(\varphi(u), \varphi(v))$ for any $u,v\in V$. Such an automorphism is called \textit{admissible}.

The flag $\mathscr{F}$ will be called \textit{almost trivial} if it consists of subspaces $0=V_0= \ldots =V_{q-1}\subset V_q=V$, the flag $\mathscr{F}$ is trivial if $q = 1$.

\textbf{Remark.} In the case of the almost trivial flag $\mathscr{F}$, the pair $(\mathscr{F}, b)$ is equivalent to the pair $(\mathscr{F}, b')$, where $b'$ has the identity matrix in any chosen basis of $V$.

We say that a decomposition of a vector space $V$ into the direct sum of its subspaces $P$ and $Q$ is coordinated with the flag $\mathscr{F}$ if $V_{j}=V_{j}\cap P+V_{j}\cap Q$ for all $j\geqslant 0$.

Since $\dim V=3$, the flag $\mathscr{F}$ consists of no more than four different subspaces. Suppose that $\{ e_{i} \}$ is enumerated in such a way that $m_{1}\leqslant m_{2}\leqslant m_{3}$. 
Consider the subspaces $V_{m_{1}}$, $V_{m_{2}}\cap V_{m_{1}}^{\perp}$, $V_{m_{3}}\cap V_{m_{2}}^{\perp} = V_{m_{2}}^{\perp}$ and choose the vector $v$ from one of these subspaces such that $b(v,v)=1$.

\begin{lemma}
	The decomposition $V=\langle v \rangle \oplus \langle v \rangle^{\bot}$ is	
	coordinated with the flag $\mathscr{F}$.
\end{lemma}
\begin{proof}
	Set $m_{0} =0$. Let $v\in V_{m_{q}}\cap V_{m_{q-1}}^{\perp}$ for some $1 \leqslant q \leqslant 3$.
	Suppose that $V_{i} = V_{m_{j}}$. If $j<q$, then $V_{i}\subseteq V_{m_{q-1}}$ and $V_{i}\subseteq \langle v \rangle^{\bot}$. 	
	If $j\geqslant q$, then $V_{m_{q}}\subseteq V_{i}$ and $\langle v \rangle \subseteq V_{i}$. Hence, for any $i$ we have $V_{i}= V_{i}\cap(\langle v \rangle \oplus \langle v \rangle^{\bot})=(\langle v \rangle\cap V_{i}) \oplus (\langle v \rangle^{\bot}\cap V_{i})$.
\end{proof}

Further, we do not assume that a basis coordinated with the flag $\mathscr{F}$ is ordered according to the heights, unless otherwise indicated.

Suppose that $\{ u, w \}$ is a basis of the space $\langle v \rangle^{\bot}$, which is coordinated with the flag $\overline{\mathscr{F}} = \{ V_{i}\cap \langle v \rangle^{\bot} \}$. Obviously, we can choose $\{ u, w \}$ in such a way that the matrix of the restriction of $b$ to $\langle v \rangle^{\bot}$ will be one of the following matrices
$$M_{0}=\begin{pmatrix} 0 & 1 \\ 1 & 0\end{pmatrix}, ~ M_{1}=\begin{pmatrix} 0 & 1 \\ 1 & 1\end{pmatrix} \text{ and } I=\begin{pmatrix} 1 & 0 \\ 0 & 1\end{pmatrix}.$$

\textbf{Remark.} In the case of the matrix $M_{1}$, if the height of $u$ is not less than that of $w$, then there exists a basis $\{u', w'\}$, which is coordinated with the flag $\overline{\mathscr{F}}$, such that the matrix of the restriction of $b$ to $\langle v \rangle^{\bot}$ will be equal to the matrix $I$.

Let $u = e_{1}$, $w = e_{2}$, $v = e_{3}$. Thus, we proved the following Lemma.
\begin{lemma}\label{lemB}
	There exists a basis $\{e_{1},e_{2}, e_{3}\}$ of the space $V$ coordinated with the flag $\mathscr{F}$ such that the matrix of $b$ is equal to one of the following matrices.
	$$B_{1}=\begin{pmatrix} 0 & 1 & 0 \\ 1 & 0 & 0 \\ 0 & 0 & 1 \end{pmatrix}, ~B_{2}=\begin{pmatrix} 0 & 1 & 0 \\ 1 & 1 & 0 \\ 0 & 0 & 1 \end{pmatrix}, ~B_{3}=\begin{pmatrix} 1 & 0 & 0 \\ 0 & 1 & 0 \\ 0 & 0 & 1 \end{pmatrix}.$$
\end{lemma}

\textbf{Remark.} In the case of the matrix $B_{2}$, if $m_{3}$ belongs to the segment $[m_{1}, m_{2}]$, then the form $b$ has the matrix $B_{1}$ with respect to the basis $\{e_{1}',e_{2}', e_{3}'\}$, $e_{1}' = e_{1}$, $e_{2}' = e_{2} + e_{3}$, $e_{3}' = e_{3} + e_{1}$, which is also coordinated with $\mathscr{F}$. Thus, we will use a basis of $V$ in which the matrix of $b$ is $B_{2}$ only in the case when $m_{1}< m_{2}$ and $m_{3} \notin [m_{1}, m_{2}]$. 

In view of this remark, we will say that the matrix of $b$ with respect to a
basis $\{e_i\}$ has a \textit{canonical} form if it is equal to one of the matrices $B_{1}$, $B_{2}$, $B_{3}$. Thus, the matrix $B_{2}$ has a canonical form if $m_{1}< m_{2}$ and $m_{3} \notin [m_{1}, m_{2}]$. In the case of the almost trivial flag we will say that the matrix of $b$ has a canonical form if it is equal to $B_{3}$.

\begin{theorem}\label{lemm5.2}
	Let $b$ and $b'$ be nondegenerate non-alternating symmetric forms on $V$, $\mathscr{F}$ and $\mathscr{F}'$ be flags of $V$. Suppose that $\{e_{1},e_{2}, e_{3}\}$ $(\{e_{1}',e_{2}', e_{3}'\})$ is a basis of $V$, which is coordinated with the flag $\mathscr{F}$ $(\mathscr{F}')$, such that the matrix of $b$ $(b')$ has a canonical form. The following statements are true.
	
	(1) If the forms $b$ and $b'$ have the matrices equal to $B_{1}$, then the pairs $(\mathscr{F}, b)$ and $(\mathscr{F}', b')$ are equivalent if and only if $\{ m_{1}, m_{2} \} = \{ m'_{1}, m'_{2} \}$ and $m_{3}= m'_{3}$.
	
	(2) If the forms $b$ and $b'$ have the matrices equal to $B_{2}$, then the pairs $(\mathscr{F}, b)$ and $(\mathscr{F}', b')$ are equivalent if and only if $m_{i}= m'_{i}$, $i=1,2,3$.
	
	(3) If the forms $b$ and $b'$ have the matrices equal to $B_{3}$, then the pairs $(\mathscr{F}, b)$ and $(\mathscr{F}', b')$ are equivalent if and only if $\{ m_{1}, m_{2}, m_{3} \} = \{ m'_{1}, m'_{2}, m_{3}' \}$.
	
	(4) If the forms $b$ and $b'$ have different matrices, then the pairs $(\mathscr{F}, b)$ and $(\mathscr{F}', b')$ are not equivalent.
\end{theorem}
\begin{proof}	
	(4) Since $\dim V =3$, $V= V^{0} \oplus (V^{0})^{\bot}$. Let $V^{1} = (V^{0})^{\bot}$. In the case of the full flag, for the form with the matrix $B_{1}$ the decomposition $V= V^{0} \oplus V^{1}$ is coordinated with the flag, but for the form with the matrix $B_{2}$ or with the matrix $B_{3}$ it is not true. Obviously, $n_{r} =\dim V_{s_{r}}\cap V_{s_{r-1}}^{\perp} - \dim V_{s_{r}}\cap V_{s_{r-1}}^{\perp}\cap V^{0}$, where $s_{1} <s_{2} <s_{3}\in \{ m_{1}, m_{2}, m_{3} \}$, are invariants (similar to invariants $n_{qq}^{1}$ from \cite[Section~1]{dep}). For the form $b$ with the matrix $B_{3}$ we have $n_{1}= n_{2}= n_{3}= 1$, but for the form $b$ with the matrix $B_{2}$ there is $i$ such that $n_{i}= 0$.
	
	If $m_{i}= m_{j}\neq m_{k}$, where $\{ i, j, k \} = \{ 1, 2, 3 \}$, then for the form with the matrix $B_{1}$ the decomposition $V= V^{0} \oplus V^{1}$ is coordinated with the flag, but for the form with the matrix $B_{3}$ it is not true.
	
	(1-3) If the conditions are fulfilled, then the isomorphism which maps $e_{i}$ into $e'_{i}$ (up to permutation) is an equivalence.
	
	Let in the cases (1) and (2) $\{ m_{1}, m_{2}, m_{3} \} = \{ m'_{1}, m'_{2}, m_{3}' \}$, but the conditions do not hold. Then $n_{i}\neq n_{i}'$ for some $i$. Consequently, the pairs are not equivalent. The statement (3) is trivial.
\end{proof}

\section{Non-alternating Hamiltonian forms with non-constant coefficients in three variables}

Put 
$$\overline{x}_{i}=x_{i}^{(2^{m_{i}}-1)}, \quad \langle x_{1}, x_{2}, x_{3} \rangle=E, \quad m_{i}=\min\{ j ~|~ x_{i}\notin E_{j} \}. $$

Let $\omega$ be the non-alternating Hamiltonian form,
$$\omega= \omega(0) + d\varphi + b_{12}\overline{x}_{1}\overline{x}_{2}dx_{1}dx_{2} + b_{13}\overline{x}_{1}\overline{x}_{3}dx_{1}dx_{3} + b_{23}\overline{x}_{2}\overline{x}_{3}dx_{2}dx_{3},$$ 
where $b_{ij}\in K$, $\varphi \in \mathfrak{m}^{(2)}S\Omega^{1}(\mathscr{F})$, and $\omega(0)$ is equal to one of the following forms (see Lemma~\ref{lemB} and Theorem~\ref{lemm5.2})
\begin{eqnarray}\label{svost1}
\omega(0)= dx_{1}dx_{2} + dx_{3}^{(2)},
\end{eqnarray}
\begin{eqnarray}\label{svost2}
\omega(0)= dx_{1}dx_{2} +dx_{2}^{(2)}+ dx_{3}^{(2)},
\end{eqnarray}
\begin{eqnarray}\label{svost3}
\omega(0)= dx_{1}^{(2)}+ dx_{2}^{(2)} + dx_{3}^{(2)}.
\end{eqnarray}

Here $dx_{i}^{(2)}$ means $(dx_{i})^{(2)}$. Recall that in the case of (\ref{svost2}), $m_{1}< m_{2}$ and $m_{3} \notin [m_{1}, m_{2}]$. 

\begin{theorem}\label{thth om}
	Let $\omega$ be a non-alternating Hamiltonian form in three variables of heights $m_{1},m_{2},m_{3}$. If the corresponding flag $\mathscr{F}$ is not almost trivial, then $\omega$ is equivalent to only one of the following forms
	$$\omega_{1}= dx_{1}dx_{2} + dx_{3}^{(2)},$$
	$$\omega_{2}= dx_{1}dx_{2} +dx_{2}^{(2)}+ dx_{3}^{(2)} ~~(m_{1}<m_{2}, m_{3}\notin [m_{1}, m_{2}]),$$
	$$\omega_{3}= dx_{1}^{(2)}+ dx_{2}^{(2)} + dx_{3}^{(2)},$$
	$$\omega_{4}= dx_{1}dx_{2} + dx_{3}^{(2)} + \overline{x}_{1}x_{3}dx_{1}dx_{3} ~~(m_{3}=1).$$
	The two forms $\omega_{i}$, $\omega_{j}$, $i\neq j$ are not equivalent.		
	If $\mathscr{F}$ is almost trivial but non-trivial, then $\omega$ is equivalent to the form $\omega_{3}$. 	
	If $\mathscr{F}$ is trivial, then $\omega$ is equivalent to one of the forms $\omega_{3}$, $\omega_{4}$, with the forms $\omega_{3}$ and $\omega_{4}$ being non-equivalent.
\end{theorem}
\begin{proof}
	By Theorem~\ref{thth 3.1} and Theorem~\ref{thth 3.2} if $m_{3}>1$ in the case of (\ref{svost1}) and $(m_{1}, m_{2}, m_{3})\neq (1,1,1)$ in the case of (\ref{svost3}), then the form $\omega$ is equivalent to the form $\omega(0)$. In the case of (\ref{svost2}), the form $\omega$ is equivalent to the form $\omega(0)$.
	
	Suppose that $\omega$ is a form such that $\omega(0)$ is of type (\ref{svost1}) and $m_{3}=1$. Theorem~\ref{thth 3.1} and Theorem~\ref{thth 3.2} imply that $\omega$ is equivalent to the form $\omega(0) +  c_{13}\overline{x}_{1}x_{3}dx_{1}dx_{3} + c_{23}\overline{x}_{2}x_{3}dx_{2}dx_{3}$.
	
	If $c_{13}\neq 0$, $c_{23} \neq 0$ and $m_{1} \leqslant m_{2}$, then the admissible automorphism $x_{1} \mapsto x_{1} + (\tilde{c}_{23}/\tilde{c}_{13})x_{2}^{(2^{m_{2} - m_{1}})}$, where $(\tilde{c}_{23})^{2^{m_{1}}} = c_{23}$, $(\tilde{c}_{13})^{2^{m_{1}}} = c_{13}$, $x_{i} \mapsto x_{i}$, $i= 2,3$ will lead to vanishing $c_{23}\overline{x}_{2}x_{3}dx_{2}dx_{3}$. If $m_{1} > m_{2}$, then the admissible automorphism $x_{2} \mapsto x_{2} + (\tilde{c}_{13}/\tilde{c}_{23})x_{1}^{(2^{m_{1} - m_{2}})}$, where $(\tilde{c}_{23})^{2^{m_{2}}} = c_{23}$, $(\tilde{c}_{13})^{2^{m_{2}}} = c_{13}$, $x_{i} \mapsto x_{i}$, $i= 1,3$ will lead to vanishing $c_{13}\overline{x}_{1}x_{3}dx_{1}dx_{3}$.
	
	Consequently, a form $\omega$ is equivalent to the form $dx_{1}dx_{2} + dx_{3}^{(2)} + c\overline{x}_{1}x_{3}dx_{1}dx_{3}$.
	Obviously, the two forms $\omega(0)+ c\overline{x}_{1}x_{3}dx_{1}dx_{3}$ and $\omega(0)+ c\overline{x}_{2}x_{3}dx_{2}dx_{3}$, where $\omega(0)= dx_{1}dx_{2} + dx_{3}^{(2)}$ are equivalent. The corresponding admissible isomorphism $\sigma\colon O(3,(m_{1},m_{2},1))\rightarrow O(3,(m_{2},m_{1},1))$ is induced by the permutation $x_{1}\mapsto x_{2}$, $x_{2}\mapsto x_{1}$, $x_{3}\mapsto x_{3}$. Furthermore, if $c\neq 0$, then the admissible automorphism $\tilde{c}x_{1} \mapsto x_{1}$, $x_{2} \mapsto \tilde{c}x_{2}$, where $(\tilde{c})^{2^{m_{1}}} = c$, $x_{3} \mapsto x_{3}$ will lead to $c=1$.
	
	Suppose that $\omega$ is a form such that $\omega(0)$ is of type (\ref{svost3}) and $(m_{1},m_{2},m_{3})=(1,1,1)$. In the case of the trivial flag, any linear isomorphism is admissible and $\omega$ is equivalent to $\omega'$ with $\omega'(0)$ of type (\ref{svost1}).
	Consequently, a form $\omega$ is equivalent to the form $\omega'' = dx_{1}dx_{2} + dx_{3}^{(2)} + cx_{1}x_{3}dx_{1}dx_{3}$. If $c\neq 0$, then $c= 1$. Note that in this case the matrix of $\omega''(0)$ does not have the canonical form. If $c= 0$, then we say that $\omega$ is equivalent to the form $dx_{1}^{(2)} + dx_{2}^{(2)} + dx_{3}^{(2)}$.
	
	Suppose that $\mathscr{F}$ is not almost trivial and show that the forms $\omega_{i}$, $\omega_{j}$, $i\neq j$ are not equivalent, $i, j= 1,2,3,4$. The fact that the forms $\omega_{i}$, $i=1,2,3$ are not equivalent follows from Theorem~\ref{lemm5.2}. 
	
	Let $\varphi$ be an admissible isomorphism, $\varphi\colon O(\mathscr{F}) \rightarrow O(\mathscr{F}')$, $\varphi_{0}$ be the linear part of $\varphi$, $\varphi(x_{i}) \equiv \varphi_{0}(x_{i})(\text{mod }\mathfrak{m}^{(2)})$.  Since $(\varphi(\omega))(0) = \varphi_{0}(\omega(0))$, we conclude that if linear parts of the forms $\omega$, $\omega'$ are not equivalent, then $\omega$ and $\omega'$ are not equivalent, either. So, all what remains to be considered is the pair $\omega_{1}$, $\omega_{4}$. Suppose that the forms $\omega_{1}$ and $\omega_{4}$ are equivalent. Hence, $\omega_{1}(0) = dx_{1}dx_{2} + dx_{3}^{(2)}$ is equivalent to $\omega_{4}(0) = dx_{1}'dx_{2}' + (dx_{3}')^{(2)}$, where $\{x_{1}, x_{2}, x_{3}\}$ and $\{x_{1}', x_{2}', x_{3}'\}$ are bases of $E$ coordinated with the flags $\mathscr{F}$, $\mathscr{F}'$, respectively. By Theorem~\ref{lemm5.2} $\{ m_{1}, m_{2} \} = \{ m'_{1}, m'_{2} \}$ and $m_{3}= m'_{3}$. Since $m_{3}'= 1$, we have $m_{3}= 1$. Theorem~\ref{simp} of the next section implies that $P(3,\overline{m},\omega_{4})$ is a simple Lie algebra and $P(3,\overline{m},\omega_{1})$ is not a simple Lie algebra.
	
	If $\mathscr{F}$ is almost trivial and $\overline{m}\neq (1,1,1)$, then $\omega(0)= \omega_{3}$ and, as it has been shown before, $\omega$ is equivalent to $\omega_{3}$. If $\overline{m}= (1,1,1)$, then the matrix of $\omega_{4}(0)$ does not have the canonical form. The fact that the forms $\omega_{3}$ and $\omega_{4}$ are not equivalent follows from Theorem~\ref{simp} of the next section according to which $P(3,\overline{1},\omega_{4})$ is a simple Lie algebra and $P(3,\overline{1},\omega_{3})$ is not a simple Lie algebra.
\end{proof}

\textbf{Remarks.} 1. The conditions under which the forms $\omega_{i}'$, $\omega_{i}''$ of the same type are equivalent for $i=1,2,3$ are obtained in Theorem~\ref{lemm5.2}. The problem of equivalence of the forms $\omega_{4}'$, $\omega_{4}''$ is not being treated in this paper, since some other technique is required.

2. Let $a\in K\diagdown \{0\}$, then the form $a\omega_{i}$ is equivalent to $\omega_{i}$, $i=1,2,3,4$. It is obvious for $i=1,2,3$. Let $\varphi$ be the admissible automorphism of $O(3, (m_{1},m_{2},1))$ such that $\varphi(x_{i})= c_{i}x_{i}$, $c_{1}=1$, $c_{2}=a$, $c_{3}=\sqrt{a}$. Then $\varphi(\omega_{4})= a\omega_{4}$.

\section{Non-alternating Hamiltonian Lie algebras in three variables}

In this section we will prove the fact that the Lie algebras which correspond to the nonequivalent forms are not isomorphic.

The algebra $P(3,\overline{m},\omega)$ is identified with the algebra $O(\mathscr{F})\big/ K \cong \mathfrak{m}$ with the Poisson bracket, corresponding to the form $\omega$ (see (\ref{eqeq 0.7})):
$$\{f,g\}_{1}= \partial_{1}f\partial_{2}g +\partial_{2}f\partial_{1}g + \partial_{3}f\partial_{3}g,$$
$$\{f,g\}_{2}= \partial_{1}f\partial_{2}g +\partial_{2}f\partial_{1}g + \partial_{1}f\partial_{1}g +\partial_{3}f\partial_{3}g,$$
$$\{f,g\}_{3}= \partial_{1}f\partial_{1}g+ \partial_{2}f\partial_{2}g + \partial_{3}f\partial_{3}g,$$
$$\{f,g\}_{4}= \partial_{1}f\partial_{2}g +\partial_{2}f\partial_{1}g + \partial_{3}f\partial_{3}g + \overline{x}_{1}x_{3}(\partial_{2}f\partial_{3}g +\partial_{3}f\partial_{2}g).$$

If $0\neq f\in \mathfrak{m}$, let $\lambda(f)$ be the nonzero homogeneous part of $f$ of the least degree.

\begin{theorem}\label{simp}
	Let $L= P(3,\overline{m},\omega_{i})$, $i=1,2,3,4$. 
	If $m_{3}>1$ in the case of $i=1$ and $\overline{m}\neq (1,1,1)$ in the case of $i=3$, then $L$ is a simple Lie algebra of dimension $2^{m_{1}+ m_{2}+ m_{3}}-1$. 
	If $m_{3}=1$ and $\overline{m}\neq (1,1,1)$, then $P^{(1)}(3,\overline{m},\omega_{1})$ is a simple Lie algebra of dimension $2^{m_{1}+ m_{2}+ 1}-2$. If $\overline{m}= (1,1,1)$, then $P^{(1)}(3,\overline{m},\omega_{3})$ is not simple.
\end{theorem}
\begin{proof}
	The algebras $P(3,\overline{m},\omega_{i})$, $i=1,2,3$ are considered in \cite[Section~4]{dep}.
	
	Let $J$ be a nonzero ideal in $P(3,\overline{m},\omega_{4})$ and $0\neq f\in J$. 	 	
	We have $ad\,x_{1}=\partial_{2}$, $ad\,x_{2}= \partial_{1} + \overline{x}_{1}x_{3}\partial_{3}$, $ad\,x_{3}=\partial_{3} + \overline{x}_{1}x_{3}\partial_{2}$. Commuting $f$ consecutively, if necessary, with $x_{1}$, $x_{3}$, $x_{2}$ we obtain $x_{1}\in J$. Now $\{x_{1},x_{2}x_{3}\}=x_{3}\in J$, $\{x_{3},x_{2}x_{3}\}=x_{2}\in J$. Let $L = \text{gr}\,P(3,\overline{m},\omega_{4}) = \bigoplus\limits_{i=-1}^{r} L_{i}$. Since $L_{-1}\subset \text{gr}\,J$, we have $L_{i}\subset \text{gr}\,J$ for $i<r$. Therefore, for any monomial $g \neq \overline{x}_{1}\overline{x}_{2}x_{3}$ there is $h$ such that $g +h\in J$, $\deg \lambda(h) > \deg g$ or $h =0$.	
	Since $\{x_{2},\overline{x}_{2}x_{3}\}=\overline{x}_{1}\overline{x}_{2}x_{3}\in J$, we have $J= P(3,\overline{m},\omega_{4})$, that is, the algebra is simple.
\end{proof}

Further we prove that the natural filtration of $P(3,\overline{m}, \omega_{i})$, $i=1,2,3,4$, $\overline{m} \neq (1,1,1)$ is intrinsically determined.
\begin{lemma}[\cite{N}]\label{lem:lemm1} 
	If $w_{1},\ldots,w_{k}\in P(3,\overline{m}, \omega)$ are linearly dependent, the same is true for $\{\lambda(w_{i})\}$.
\end{lemma}

Let $L^{i} = P(3,\overline{m},\omega_{i})$, $i=1,2,3,4$ be a non-alternating Hamiltonian Lie algebra, where $\overline{m} \neq (1,1,1)$.
For $\phi\in Der(L^{i})$ let $R^{i}(\phi)=\dim(\textrm{Im}\phi)$ (see \cite{N}). Obviously, $R^{i}(\phi)=R^{i}(a\phi)$ for any $a\in K^{*}$. If $M$ is a subalgebra of $Der(L^{i})$, let $R^{i}(M)=\min\limits_{0\neq\phi\in M}R^{i}(\phi)$. 

Let $\{L_{(i)}\}$ be the standard filtration, $\{L_{i}\}$ be the standard grading and $\xi =ad\,\bar{x} = ad\,\bar{x}_{1}\bar{x}_{2}\bar{x}_{3} =ad\,x_{1}^{(2^{m_{1}}-1)}x_{2}^{(2^{m_{2}}-1)}x_{3}^{(2^{m_{3}}-1)}$.

\begin{lemma}\label{lem:lem5.5.1} 
	Let $0\neq\phi\in\langle\xi\rangle$. If $i=2,3$ or $i=1$, $m_{3}>1$, then $R^{i}(\phi)=4$. If $i=1$, $m_{3}=1$ or $i=4$, then $R^{i}(\phi)=3$.
\end{lemma}
\begin{proof}
	Obviously, $\phi(x_{1}),\phi(x_{2}), \phi(x_{3}) \in \textrm{Im}\phi$ are linearly independent. The action of $\phi$ on $L_{(0)}$ corresponds to the one-dimensional representation of $L_{(0)}$ on $L_{(r)} = \langle \bar{x}\rangle$. Since $\phi(L_{(1)})=0$ and $\phi([L_{(0)},L_{(0)}])=0$, it follows that Im($\phi$) $= \langle\phi(x_{1}), \phi(x_{2}), \phi(x_{3}), \bar{x}\rangle$ if there exists $x_{j}^{(2)}\in L$ such that $\phi(x_{j}^{(2)})\neq0$. Otherwise, Im($\phi$) $= \langle\phi(x_{1}), \phi(x_{2}), \phi(x_{3})\rangle$.
	
	Let $i=1$. Then $\phi(x_{1}^{(2)})= \phi(x_{2}^{(2)})=0$. If $m_{3}>1$, then $\phi(x_{3}^{(2)})\neq0$. Consequently, $R^{1}(\phi)=4$. If $m_{3}=1$, then $R^{1}(\phi)=3$.
	
	Let $i=2$. The condition $m_{1}< m_{2}$, $m_{3}\notin [m_{1}, m_{2}]$ yield $m_{1}>1$ or $m_{3}>1$. Since $\phi(x_{1}^{(2)})= \phi(x_{3}^{(2)})\neq0$, we have $R^{2}(\phi)=4$.
	
	Let $i=3$. Since $\overline{m} \neq (1,1,1)$, we have $m_{j}>1$ and $\phi(x_{j}^{(2)})\neq0$. Consequently, $R^{3}(\phi)=4$.
	
	Let $i=4$. Then $m_{3}=1$ and $\phi(x_{1}^{(2)})= \phi(x_{2}^{(2)})=0$. Consequently, $R^{4}(\phi)=3$.
\end{proof}

\begin{lemma}\label{lem:lem6.1} 
	Let $i=2,3$ or $i=1$, $m_{3}>1$. If $0\neq D\in L^{i}_{t}$, $-1\leqslant t< 2^{m_{1}}+ 2^{m_{2}}+ 2^{m_{3}} -5$, then there exist $5$ homogeneous elements $E_{1}, E_{2} ,E_{3}, E_{4}, E_{5}\in L^{i}$ such that $\{D,E_{1}\}, \{D,E_{2}\},\ldots, \{D,E_{5}\}$ are linearly independent.
\end{lemma}
\begin{proof}
	\underline{The case of} $L^{1}$, $m_{3}>1$.
	
	$(1)~ t=-1$. Let $D=a_{1}x_{1} + a_{2}x_{2} + a_{3}x_{3}$ and $a_{3}\neq0$. Put $E_{1}=x_{3}^{(2)}$, $E_{2}=x_{3}^{(3)}$, $E_{3}=x_{1}x_{3}^{(3)}$, $E_{4}=x_{2}x_{3}^{(3)}$, $E_{5}=x_{1}x_{2}x_{3}^{(3)}$. 	
	Let $D=a_{1}x_{1} + a_{2}x_{2}$ and $a_{i}\neq0$. Put $E_{1}=x_{1}x_{2}x_{3}$, $E_{2}=x_{1}x_{2}x_{3}^{(3)}$, $E_{3}=x_{j}x_{3}$, $E_{4}=x_{j}x_{3}^{(2)}$, $E_{5}=x_{j}x_{3}^{(3)}$, where $j\neq i$.
	
	$(2)~ t>-1$. Let $D= ax_{1}^{(\alpha)}x_{2}^{(\beta)}x_{3}^{(\gamma)} + f(x_{1},x_{2},x_{3})$, where $a\neq0$, $\alpha,\beta>0$ and $f$ contains $x_{3}^{(s)}g(x_{1},x_{2})$ if $s\leqslant \gamma$. Put $E_{1}= x_{1}^{(2^{m_{1}}-\alpha)}$, $E_{2}= x_{3}^{(2^{m_{3}}+1-\gamma)}$ ($\gamma>1$) or $x_{1}x_{3}^{(2)}$.
	If $\gamma< 2^{m_{3}}-2$, then $E_{3} = x_{2}^{(2^{m_{2}}-\beta)}$, $E_{4}= x_{1}^{(2^{m_{1}}-\alpha)}x_{3}^{(2^{m_{3}}-1-\gamma)}$, $E_{5}= x_{1}^{(2^{m_{1}}-\alpha)}x_{3}^{(2^{m_{3}}-2-\gamma)}$ ($\gamma$ is even) or $x_{3}^{(2)}$ ($\gamma$ is odd). If $\gamma= 2^{m_{3}}-2$, then $E_{3} = x_{2}$, $E_{4} = x_{1}^{(2^{m_{1}}-\alpha)}x_{3}$, $E_{5} = x_{2}^{(2^{m_{2}}-\beta)}x_{3}$. If $\gamma= 2^{m_{3}}-1$, then $E_{3} = x_{2}^{(2^{m_{2}}-\beta)}$, $E_{4} = x_{1}^{(2^{m_{1}}-1-\alpha)}x_{2}^{(2^{m_{2}}-1-\beta)}x_{3}$, $E_{5} = x_{1}^{(2^{m_{1}}-1-\alpha)}x_{2}^{(2^{m_{2}}-1-\beta)}x_{3}^{(2)}$.
	
	Let $D= ax_{3}^{(t+2)}+ f(x_{1},x_{3})+ g(x_{2},x_{3})$ and $a\neq0$. Put $\{E_{1},\ldots, E_{5}\} = \left\lbrace x_{3}, \bar{x}_{1}x_{3}, \bar{x}_{2}x_{3}, x_{1}x_{2}x_{3},\right. \\ \left.x_{1}x_{2}x_{3}^{(2^{m_3}-1-t)}\tiny\right\rbrace$.
	
	Let $D= ax_{i}^{(\alpha)}x_{3}^{(\gamma)} + f(x_{1},x_{3})+ g(x_{2},x_{3})$, where $a\neq0$, $\alpha>0$ and $f,g$ contain $x_{3}^{(s)}x_{q}^{(h)}$ if $s\leqslant \gamma$, $h>0$. If $\gamma>0$, then  $E_{1} = x_{j}$, $E_{2}= x_{i}^{(2^{m_{i}}-1-\alpha)}x_{3}$, $E_{3}= x_{i}^{(2^{m_{i}}-\alpha)}x_{j}$, $E_{4}= x_{i}^{(2^{m_{i}}-\alpha)}x_{j}x_{3}^{(2^{m_{3}}-1-\gamma)}$ ($\gamma<2^{m_{3}}-1$) or $x_{i}^{(2^{m_{i}}-1-\alpha)}x_{j}^{(2^{m_{j}}-2)}x_{3}$, $E_{5}= x_{j}x_{3}^{(2)}$ ($\gamma$ is odd) or $x_{j}x_{3}^{(2^{m_{3}}-2-\gamma)}$ ($\gamma<2^{m_{3}}-2$ is even) or $x_{j}x_{3}$, where $\{i, j\} = \{1, 2\}$. If $\gamma=0$, then $\{E_{1},\ldots, E_{5}\} = \{x_{j}, x_{i}^{(2^{m_{i}}-\alpha)}x_{j}, x_{i}^{(2^{m_{i}}-\alpha)}x_{j}x_{3}, x_{i}^{(2^{m_{i}}-\alpha)}x_{j}x_{3}^{(2)}, x_{i}^{(2^{m_{i}}-\alpha)}x_{j}x_{3}^{(3)}\}$.
	
	\underline{The case of} $L^{2}$, $m_{3}>1$.
	
	$(1)~ t=-1$. Let $D=a_{1}x_{1} + a_{2}x_{2} + a_{3}x_{3}$ and $a_{3}\neq0$. The same $L^{1}$. 	
	Let $D=a_{1}(x_{1}+x_{2}) + a_{2}x_{2}$ and $a_{i}\neq0$. Put $\{E_{1},\ldots, E_{5}\}$ the same $L^{1}$.
	
	$(2)~ t>-1$. Let $D= ax_{1}^{(\alpha)}x_{2}^{(\beta)}x_{3}^{(\gamma)} + f(x_{1},x_{2},x_{3})$, where $a\neq0$, $\alpha,\beta>0$ and $f$ contains $x_{3}^{(s)}g(x_{1},x_{2})$ if $s\leqslant \gamma$. Put $E_{1}= x_{2}^{(2^{m_{2}}-\beta)}$, $E_{2}= x_{3}^{(2^{m_{3}}+1-\gamma)}$ ($\gamma>1$) or $x_{3}^{(3)}$ ($\gamma=1$) or $x_{1}^{(2^{m_{1}}-\alpha)}x_{3}^{(2)}$ ($\gamma=0\neq 2^{m_{3}}-4$, $(\alpha,\beta)\neq (2^{m_{1}}-1,2^{m_{2}}-1)$) or $x_{1}x_{3}^{(2)}+ x_{2}x_{3}^{(2)}$.
	If $\gamma< 2^{m_{3}}-2$, then $E_{3} = x_{1}^{(2^{m_{1}}-\alpha)}$ ($(\alpha,\beta)\neq (2^{m_{1}}-1,2^{m_{2}}-1)$) or $x_{1}+x_{2}$, $E_{4}= x_{2}^{(2^{m_{2}}-\beta)}x_{3}^{(2^{m_{3}}-1-\gamma)}$, $E_{5}= x_{2}^{(2^{m_{2}}-\beta)}x_{3}^{(2^{m_{3}}-2-\gamma)}$ ($\gamma$ is even) or $x_{3}^{(2)}$ ($\gamma$ is odd). If $\gamma= 2^{m_{3}}-2$, then $E_{3} = x_{1}+ x_{2}$, $E_{4} = x_{2}^{(2^{m_{2}}-\beta)}x_{3}$, $E_{5}  = x_{1}^{(2^{m_{1}}-\alpha)}x_{3}$ ($(\alpha,\beta)\neq (2^{m_{1}}-1,2^{m_{2}}-1)$) or $x_{1}x_{3}+x_{2}x_{3}$. If $\gamma= 2^{m_{3}}-1$, then $E_{3} = x_{1}^{(2^{m_{1}}-\alpha)}$, $E_{4} = x_{1}^{(2^{m_{1}}-1-\alpha)}x_{2}^{(2^{m_{2}}-1-\beta)}x_{3}$, $E_{5} = x_{1}^{(2^{m_{1}}-1-\alpha)}x_{2}^{(2^{m_{2}}-1-\beta)}x_{3}^{(2)}$.
	
	Let $D= ax_{3}^{(t+2)}+ f(x_{1},x_{3})+ g(x_{2},x_{3})$ and $a\neq0$. The same $L^{1}$.
	
	Let $D= ax_{i}^{(\alpha)}x_{3}^{(\gamma)} + f(x_{1},x_{3})+ g(x_{2},x_{3})$, where $a\neq0$, $\alpha>0$ and $f,g$ contain $x_{3}^{(s)}x_{q}^{(h)}$ if $s\leqslant \gamma$, $h>0$. The same $L^{1}$.
	
	\underline{The case of} $L^{2}$, $m_{3}=1$, $m_{1}>1$, $m_{2}>m_{1}>1$.
	
	$(1)~ t=-1$. Let $D=a_{1}(x_{1}+ x_{2}) + a_{2}x_{2} + a_{3}x_{3}$ and $a_{i}\neq0$, $i=1,2$. Put $E_{1}=x_{j}^{(2)}$, $E_{2}=x_{j}^{(3)}$, $E_{3}=x_{j}^{(3)}x_{3}$, $E_{4}=x_{i}x_{j}^{(3)}$, $E_{5}=x_{i}x_{j}^{(3)}x_{3}$, where $j\neq i$. 	
	Let $D=a_{3}x_{3}$ and $a_{3}\neq0$. Put $E_{1}=x_{1}x_{3}$, $E_{2}=x_{2}x_{3}$, $E_{3}=x_{1}^{(2)}x_{3}$, $E_{4}=x_{2}^{(2)}x_{3}$, $E_{5}=x_{1}x_{2}x_{3}$.
	
	$(2)~ t>-1$. Let $D= ax_{1}^{(\alpha)}x_{2}^{(\beta)}x_{3} + f(x_{1},x_{2},x_{3})$, where $a\neq0$, $\beta>0$ and $f$ contains $x_{1}^{(s)}g(x_{2},x_{3})$ if $s\leqslant \alpha$. Put $E_{1}= x_{2}^{(2^{m_{2}}-1-\beta)}x_{3}$.
	If $\gamma< 2^{m_{3}}-1$, then $E_{2} = x_{1}^{(2^{m_{1}}-1-\alpha)}x_{2}^{(2^{m_{2}}-1-\beta)}x_{3}$, $E_{3}= x_{1}^{(2^{m_{2}}+1-\alpha)}$ ($\alpha>1$) or $x_{1}^{(2)}x_{3}$. If $\gamma< 2^{m_{3}}-2$, then $E_{4}= x_{3}$ ($\beta\neq 2^{m_{2}}-1$) or $x_{1}x_{2}$ ($\beta= 2^{m_{2}}-1$, $\alpha$ is even) or $x_{1}^{(2^{m_{1}}-1-\alpha)}x_{2}$ ($\beta= 2^{m_{2}}-1$, $\alpha$ is odd), $E_{5} = x_{1}^{(2^{m_{1}}-2-\alpha)}x_{2}^{(2^{m_{2}}-1-\beta)}x_{3}$ ($\alpha$ is even) or $x_{1}^{(2)}$ ($\alpha$ is odd). If $\alpha= 2^{m_{3}}-2$, then $E_{4} = x_{1}+ x_{2}$, $E_{5}  = x_{1}x_{3}$ ($\beta\neq 2^{m_{2}}-1$) or $x_{1}x_{2}$. If $\gamma= 2^{m_{3}}-1$, then $E_{2} = x_{2}$, $E_{3} = x_{3}$, $E_{4} = x_{1}^{(2)}$ ($\beta\neq 2^{m_{2}}-2$) or $x_{1}x_{2}$, $E_{5} = x_{1}^{(2)}x_{2}^{(2^{m_{2}}-1-\beta)}$.
	
	Let $D= ax_{1}^{(t+2)}+ f(x_{1},x_{2})+ g(x_{1},x_{3})$ and $a\neq0$. Put $\{E_{1},\ldots, E_{5}\} = \left\lbrace x_{1}, x_{1}\bar{x}_{2}, x_{1}x_{3}, x_{1}x_{2}x_{3},\right. \\ \left.x_{1}^{(2^{m_1}-1-t)}x_{2}x_{3}\right\rbrace$.
	
	Let $D= ax_{1}^{(\alpha)}x_{3} + f(x_{1},x_{2})+ g(x_{1},x_{3})$, where $a\neq0$, $\alpha>0$ and $f,g$ contain $x_{1}^{(s)}x_{q}^{(h)}$ if $s\leqslant \alpha$, $h>0$. Put $\{E_{1},\ldots, E_{5}\} = \{x_{3}, x_{2}, x_{2}^{(2)}, x_{2}^{(3)}, \bar{x}_{2}x_{3}\}$.
	
	Let $D= ax_{1}^{(\alpha)}x_{2}^{(\beta)} + f(x_{1},x_{2})+ g(x_{1},x_{3})$, where $a\neq0$, $\beta>0$ and $f,g$ contain $x_{1}^{(s)}x_{q}^{(h)}$ if $s\leqslant \alpha$, $h>0$. If $\alpha>0$, then $E_{1}= x_{1}^{(2^{m_{1}}-\alpha)}$ ($(\alpha,\beta)\neq (2^{m_{1}}-1,2^{m_{2}}-1)$) or $x_{1}+x_{2}$, $E_{2}= x_{2}^{(2^{m_{2}}-\beta)}$, $E_{3}= x_{2}^{(2^{m_{2}}-\beta)}x_{3}$, $E_{4}= x_{1}^{(2^{m_{1}}+1-\alpha)}$ ($\alpha>1$) or $x_{1}^{(2)}x_{2}^{(2^{m_{2}}-1-\beta)}$, $E_{5}= x_{2}^{(2)}x_{3}$ ($\alpha$ is odd) or $x_{1}^{(2^{m_{1}}-1-\alpha)}x_{3}$ ($\alpha$ is even, $(\alpha,\beta)\neq (2^{m_{1}}-2,2^{m_{2}}-1)$) or $x_{1}x_{3}+ x_{2}x_{3}$.
	If $\alpha=0$, then $\{E_{1},\ldots, E_{5}\} = \{x_{1}, x_{1}x_{3}, x_{1}^{(2)}, x_{1}^{(3)}, x_{1}x_{2}^{(2^{m_{2}}-\beta)}\}$.
	
	\underline{The case of} $L^{3}$. Without the loss of generality we can assume that $m_{3}>1$.
	
	$(1)~ t=-1$. The same $L^{1}$ with $i=j$.
	
	$(2)~ t>-1$. Let $D= ax_{1}^{(\alpha)}x_{2}^{(\beta)}x_{3}^{(\gamma)} + f(x_{1},x_{2},x_{3})$, where $a\neq0$, $\alpha,\beta>0$ and $f$ contains $x_{3}^{(s)}g(x_{1},x_{2})$ if $s\leqslant \gamma$. Put $E_{1}= x_{1}x_{2}^{(2^{m_{2}}-1-\beta)}$, $E_{2}= x_{3}^{(2^{m_{3}}+1-\gamma)}$ ($\gamma>1$) or $x_{2}x_{3}^{(2)}$.
	If $\gamma< 2^{m_{3}}-2$, then $E_{3} = x_{1}^{(2^{m_{1}}-1-\alpha)}x_{2}$ ($(\alpha,\beta)\neq (2^{m_{1}}-2,2^{m_{2}}-2)$) or $x_{1}x_{3}^{(2^{m_{3}}-1-\gamma)}$, $E_{4}= x_{1}x_{2}^{(2^{m_{2}}-1-\beta)}x_{3}^{(2^{m_{3}}-1-\gamma)}$, $E_{5}= x_{1}x_{2}^{(2^{m_{2}}-1-\beta)}x_{3}^{(2^{m_{3}}-2-\gamma)}$ ($\gamma$ is even) or $x_{3}^{(2)}$ ($\gamma$ is odd). If $\gamma= 2^{m_{3}}-2$, then $E_{3} = x_{2}$, $E_{4} = x_{1}x_{3}$, $E_{5} = x_{1}x_{2}^{(2^{m_{2}}-1-\beta)}x_{3}$. Let $\gamma= 2^{m_{3}}-1$. Since $\alpha,\beta>0$ and $(\alpha,\beta)\neq (2^{m_{1}}-1,2^{m_{2}}-1)$, we have $m_{1}>1$ or $m_{2}>1$. Let $m_{1}>1$, then $E_{3} = x_{1}^{(2^{m_{1}}-1-\alpha)}x_{2}$ ($(\alpha,\beta)\neq (2^{m_{1}}-2,2^{m_{2}}-2)$) or $x_{1}x_{3}^{(2)}$, $E_{4} = x_{1}^{(2^{m_{1}}-1-\alpha)}x_{2}^{(2^{m_{2}}-1-\beta)}x_{3}^{(2)}$, $E_{5} = x_{1}x_{3}$ ($\alpha$ is even) or $x_{1}^{(2)}x_{3}$ ($\alpha$ is odd).
	
	Let $D= ax_{3}^{(t+2)}+ f(x_{1},x_{3})+ g(x_{2},x_{3})$ and $a\neq0$. The same $L^{1}$.
	
	Let $D= ax_{i}^{(\alpha)}x_{3}^{(\gamma)} + f(x_{1},x_{3})+ g(x_{2},x_{3})$, where $a\neq0$, $\alpha>0$ and $f,g$ contain $x_{3}^{(s)}x_{q}^{(h)}$ if $s\leqslant \gamma$, $h>0$. If $\gamma>0$, then  $E_{1} = x_{i}$, $E_{2}= x_{i}^{(2^{m_{i}}-1-\alpha)}x_{3}$, $E_{3}= x_{1}x_{2}$, $E_{4}= x_{1}x_{2}x_{3}^{(2^{m_{3}}-1-\gamma)}$ ($\gamma<2^{m_{3}}-1$) or $x_{i}^{(2^{m_{i}}-1-\alpha)}x_{j}^{(2^{m_{j}}-2)}x_{3}$, $E_{5}= x_{j}x_{3}^{(2)}$ ($\gamma$ is odd) or $x_{i}x_{3}^{(2^{m_{3}}-2-\gamma)}$ ($\gamma<2^{m_{3}}-2$ is even) or $x_{i}x_{3}$ ($\gamma=2^{m_{3}}-2$, $\alpha\neq2^{m_{i}}-2$) or $x_{i}^{(3)}$, where $\{i, j\} = \{1, 2\}$. If $\gamma=0$, then $\{E_{1},\ldots, E_{5}\} = \{x_{i}, x_{i}^{(2^{m_{i}}+1-\alpha)}, x_{i}^{(2^{m_{i}}+1-\alpha)}x_{3}, x_{i}^{(2^{m_{i}}+1-\alpha)}x_{3}^{(2)}, x_{i}^{(2^{m_{i}}+1-\alpha)}x_{3}^{(3)}\}$.
\end{proof}

\begin{lemma} 
	Let $f\in L^{4}_{(t)}\diagdown L^{4}_{(t+1)}$, $g\in L^{4}_{(s)}\diagdown L^{4}_{(s+1)}$ and $\{\lambda(f),g\}\neq 0$.
	
	(1) If $g=x_{1}^{(\alpha)}x_{2}^{(\beta)}x_{3}^{(\gamma)}$ and $\alpha>0$, then $\lambda(\{f,g\}) = \{\lambda(f),g\}$. 
	
	(2) If $g=x_{2}^{(\beta)}x_{3}^{(\gamma)}$ and $\{\lambda(f), g\}\in L^{4}_{(t+s)}\diagdown L^{4}_{(t+s+1)}$, then $\lambda(\{f,g\}) = \{\lambda(f),g\}$. 
\end{lemma}
\begin{proof}
	Recall, that $\{f,g\}= \partial_{1}f\partial_{2}g +\partial_{2}f\partial_{1}g + \partial_{3}f\partial_{3}g + \overline{x}_{1}x_{3}(\partial_{2}f\partial_{3}g +\partial_{3}f\partial_{2}g)$. 
	
	(1) Obviously, if $g=x_{1}^{(\alpha)}x_{2}^{(\beta)}x_{3}^{(\gamma)}$ and $\alpha>0$, then $\{f,g\}= \partial_{1}f\partial_{2}g +\partial_{2}f\partial_{1}g + \partial_{3}f\partial_{3}g$. 
	
	(2) Let $f= \lambda(f) + f_{1}$, where $f_{1}\in L^{4}_{(t+1)}$. Then $\{f,g\}\in L^{4}_{(t+s)}$, $\{f_{1},g\} \in L^{4}_{(t+s+1)}$. Since $\{\lambda(f), g\}\in L^{4}_{(t+s)}\diagdown L^{4}_{(t+s+1)}$, $\lambda(\{f,g\}) = \lambda(\{\lambda(f),g\}+ \{f_{1},g\})= \lambda(\{\lambda(f),g\}) = \{\lambda(f),g\}$.
\end{proof}

\begin{lemma}\label{lem:lem6.2} 
	Let $i=4$ or $i=1$, $m_{3}=1$. If $0\neq D\in L^{i}_{(t)}\diagdown L^{i}_{(t+1)}$, $-1\leqslant t< 2^{m_{1}} + 2^{m_{2}}-3$, then there exist $4$ homogeneous elements $E_{1}, E_{2} ,E_{3}, E_{4}\in L^{i}$ such that $\{D,E_{1}\}, \ldots, \{D,E_{4}\}$ are linearly independent.
\end{lemma}
\begin{proof}
	We will prove that there exist $4$ homogeneous elements $E_{1}, E_{2} ,E_{3}, E_{4}\in L^{i}$ such that $\{\lambda(D),E_{1}\}, \ldots, \{\lambda(D),E_{4}\}$ are linearly independent. We will search $E_{i}$ for $L^{4}$ which satisfy Lemma~6. Thus $\{\lambda(D), E_{i}\}=\lambda(\{D,E_{i}\})$ and $\{D,E_{i}\}$ are also linearly independent by Lemma~\ref{lem:lemm1}.
	
	\underline{The case of} $L^{4}$, $m_{1}>1$.
	
	$(1)~ t=-1$. Let $\lambda(D) =a_{1}x_{1} + a_{2}x_{2} + a_{3}x_{3}$ and $a_{3}\neq0$. Put $E_{1}=x_{1}x_{3}$, $E_{2}=x_{1}x_{2}x_{3}$, $E_{3}=x_{1}^{(2)}x_{3}$, $E_{4}=x_{1}^{(3)}x_{3}$.
	Let $\lambda(D)=a_{1}x_{1} + a_{2}x_{2}$ and $a_{i}\neq0$. Put $E_{1}=x_{1}x_{2}$, $E_{2}=x_{1}x_{2}x_{3}$, $E_{3}=x_{1}^{(2)}x_{2}$, $E_{4}=x_{1}^{(3)}x_{2}$.
	
	$(2)~ t>-1$.  Let $\lambda(D)= ax_{1}^{(\alpha)}x_{2}^{(\beta)} + f(x_{1},x_{2},x_{3})$, where $a\neq0$, $\beta>0$ and $f$ contains $x_{1}^{(s)}x_{2}^{(h)}$ if $s< \alpha$. Put $\{E_{1},\ldots, E_{4}\} = \{x_{1}, x_{1}x_{3}, x_{1}^{(2^{m_{1}}-1-\alpha)}x_{2}^{(2^{m_{2}}-\beta)}, x_{1}^{(2^{m_{1}}-1-\alpha)}x_{2}^{(2^{m_{2}}-\beta)}x_{3}\}$.
	
	Let $\lambda(D)= ax_{1}^{(\alpha)}x_{3} + f(x_{1},x_{2},x_{3})$, where $a\neq0$, $\alpha>0$ and $f$ contains $x_{1}^{(s)}x_{2}^{(h)}$ if $h=0$. Put $E_{1} = x_{3}$. If $\alpha$ is even, then $\{E_{2}, E_{3}, E_{4}\} = \{x_{1}x_{3}, \bar{x}_{2}x_{3}, x_{1}x_{2}x_{3}\}$. If $\alpha$ is odd, then $E_{2} = x_{2}x_{3}$, $E_{3} = x_{1}^{(2^{m_{1}}-1-\alpha)}x_{2}$, $E_{4} = x_{1}^{(2^{m_{1}}-1-\alpha)}x_{2}x_{3}$ ($\alpha\neq2^{m_{1}}-1$) or $x_{1}x_{2}$.
	
	Let $\lambda(D)= ax_{1}^{(\alpha)}x_{2}^{(\beta)}x_{3} + f(x_{1},x_{2})x_{3} + bx_{1}^{(t+2)}$, where $a\neq0$, $\beta>0$ and $f$ contains $x_{1}^{(s)}x_{2}^{(h)}$ if $s<\alpha$, $h>0$. Put $E_{1}= x_{1}$. If $\alpha>0$, then $E_{2}= x_{2}^{(2^{m_{2}}-\beta)}$, $E_{3}= x_{3}$, $E_{4}= x_{1}^{(2^{m_{1}}-\alpha)}$ ($\alpha\neq2^{m_{1}}-1$) or $x_{2}^{(2^{m_{2}}-\beta)}x_{3}$. If $\alpha=0$, then $\{E_{2}, E_{3}, E_{4}\} = \{x_{1}x_{2}^{(2^{m_{2}}-\beta)}, x_{1}x_{3}, \bar{x}_{1}\}$.
	
	Let $\lambda(D)= ax_{1}^{(t+2)}$, $a\neq0$. Put $\{E_{1},\ldots, E_{4}\} = \{x_{2}, x_{2}x_{3}, x_{1}^{(2^{m_{1}}-2-t)}x_{2}, x_{1}^{(2^{m_{1}}-2-t)}x_{2}x_{3}\}$.
	
	\underline{The case of} $L^{4}$, $m_{1}=1$, $m_{2}>1$.
	
	$(1)~ t=-1$. Let $\lambda(D) =a_{1}x_{1} + a_{2}x_{2} + a_{3}x_{3}$ and $a_{3}\neq0$. Put $E_{1}=x_{1}x_{3}$, $E_{2}=x_{1}x_{2}x_{3}$, $E_{3}=x_{1}x_{2}^{(2)}x_{3}$, $E_{4}=x_{1}x_{2}^{(3)}x_{3}$.
	Let $\lambda(D)=a_{1}x_{1} + a_{2}x_{2}$ and $a_{i}\neq0$. Put $E_{1}=x_{1}x_{2}$, $E_{2}=x_{1}x_{2}x_{3}$, $E_{3}=x_{1}x_{2}^{(2)}$, $E_{4}=x_{1}x_{2}^{(3)}$.
	
	$(2)~ t>-1$. Let $\lambda(D)= ax_{2}^{(t+2)} + f(x_{1},x_{2},x_{3})$ and $a\neq0$. Put $\{E_{1},\ldots, E_{4}\} = \{x_{1}, x_{1}x_{3}, x_{1}x_{2}^{(2^{m_{2}}-2-t)}, x_{1}x_{2}^{(2^{m_{2}}-2-t)}x_{3}\}$.
	
	Let $\lambda(D)= ax_{2}^{(\alpha)}x_{3} + f(x_{1},x_{2},x_{3})$, where $a\neq0$, $\alpha>0$ and $f$ does not contain $x_{2}^{(t+2)}$. Put $\{E_{1}, E_{2}, E_{3}\} = \{x_{1}, x_{1}x_{3}, x_{1}x_{2}^{(2^{m_{2}}-\alpha)}\}$ and $E_{4}= x_{1}x_{2}^{(2^{m_{2}}-1-\alpha)}x_{3}$ ($\alpha\neq2^{m_{2}}-1$) or $x_{3}$.
	
	Let $\lambda(D)=  ax_{1}x_{2}^{(t+1)} +bx_{1}x_{2}^{(t)}x_{3}$ and $a_{1}\neq0$. Put $\{E_{1},\ldots, E_{4}\} = \{x_{1}, x_{2}, x_{1}x_{3}, x_{2}x_{3}\}$.
	
	Let $\lambda(D)= ax_{1}x_{2}^{(t)}x_{3}$, $a\neq0$. Put $E_{1}=x_{2}$, $E_{2}=x_{1}x_{2}^{(2^{m_{2}}-1-t)}$. If $t>0$, then $E_{3} =x_{1}$, $E_{4}=x_{2}^{(2^{m}-t)}$. If $t=0$, then $E_{3} =x_{1}x_{2}$, $E_{4}=\bar{x}_{2}$.
	
	\underline{The case of} $L^{3}$, $m_{3}=1$. The same $L^{4}$.
\end{proof}

\begin{theorem}\label{key777}
	If $i=2,3$ or $i=1$, $m_{3}>1$, then $R^{i}(L^{i})=4$ and 	
	$R^{i}(D)= 4$ if and only if $0\neq D\in\langle \xi\rangle$. If $i=1$, $m_{3}=1$ or $i=4$, then $R^{i}(L^{i})=3$ and 	
	$R^{i}(D)= 3$ if and only if $0\neq D\in\langle \xi\rangle$.
\end{theorem}
\begin{proof}
	(1) Let $i=2,3$ or $i=1$, $m_{3}>1$. It follows from Lemma~\ref{lem:lem5.5.1} that if $0\neq D\in\langle\xi\rangle$, then $R^{i}(D)=4$. We shall prove that if $f\in L^{i}$ and $f\notin\langle \bar{x}\rangle$, then $R^{i}(ad\,f)>4$. Clearly, $\lambda(f)\notin\langle \bar{x}\rangle$, thus, according to Lemma~\ref{lem:lem6.1}, there are $5$ homogeneous elements $E_{1},\ldots,E_{5}\in L^{i}$ such that $\{\lambda(f) , E_{j}\}$, $1\leqslant j\leqslant 5$ are linearly independent. But $\{\lambda(f), E_{j}\}=\lambda(\{f,E_{j}\})$. Hence, $\{f,E_{j}\}$ are also linearly independent by Lemma~\ref{lem:lemm1}. Therefore, $R^{i}(ad\,f)\geqslant 5>4$.
	
	(2) Let $i=1$, $m_{3}=1$ or $i=4$. It follows from Lemma~\ref{lem:lem5.5.1} that if $0\neq D\in\langle\xi\rangle$, then $R^{i}(D)=3$. If $f\in L^{i}$ and $f\notin\langle \bar{x}\rangle$, then according to Lemma~\ref{lem:lem6.2}, there are $4$ homogeneous elements $E_{1},\ldots,E_{4}\in L^{i}$ such that $\{f, E_{j}\}$, $1\leqslant j\leqslant 4$ are linearly independent. Therefore, $R^{i}(ad\,f)\geqslant 4>3$.
\end{proof}

\begin{corollary}
	Let $L = P(3,\overline{m}, \omega_{i})$, $i=1,2,3,4$, $\overline{m}\neq(1,1,1)$. The standard maximal subalgebra $\mathscr{L}_{0}$ of $L$ is the normalizer of $\langle \bar{x} \rangle$ in $L$. Thus, $\mathscr{L}_{0}$ is an invariant subalgebra and the natural filtration of $L$ is intrinsically determined.
\end{corollary}

\begin{theorem}\label{key999}
	Let $L^{i} = P(3,\overline{m}, \omega_{i})$, $L^{j} = P(3,\overline{m}', \omega_{j})$, $i,j=1,2,3,4$. If $i\neq j$, then the non-alternating Hamiltonian Lie algebras $L^{i}$ and $L^{j}$ are not isomorphic.
\end{theorem}
\begin{proof}
	Suppose that $L^{i}$ and $L^{j}$ are isomorphic, $\varphi\colon L^{i}\rightarrow L^{j}$ is an isomorphism. If $\overline{m}=(1,1,1)$, then $\overline{m}'=(1,1,1)$ either.
	
	Consider the case when $\overline{m}\neq(1,1,1)$. The algebras $L^{i}$ and $L^{j}$ have the same invariant $R$, $R^{i}= R^{j}$. Let $\bar{x}(i)$ is the monomial in $L^{i}$ of maximal degree. It follows from Theorem~\ref{key777} that $\varphi(\bar{x}(i))= a\bar{x}(j)$, $a\in K^{\ast}$, and, consequently, $\varphi(L_{(0)}^{i}) = L_{(0)}^{j}$, i.e. $\varphi$ is an isomorphism of transitive Lie algebras. From Theorem~\ref{thth 5.2} and Remark~2 of Section~4 we obtain that the forms $\omega_i$, $\omega_j$ are equivalent. By Theorem~\ref{thth om}~ $i=j$.
	
	Suppose that $\overline{m}=(1,1,1)$. In this case we have only $L^{3} = P(3, (1,1,1), \omega_{3})$ which is not simple and $L^{4} = P(3, (1,1,1), \omega_{4})$ which is simple.
\end{proof}

\begin{corollary}
	Let $\mathscr{L}^{i}$ be non-alternating Hamiltonian Lie algebras in three variables, $P^{(1)}(\mathscr{F}_{i}, \omega_{i})\subseteq \mathscr{L}^{i} \subseteq \widetilde{P}(\mathscr{F}_{i}, \omega_{i})$, $i=1,2,3,4$. If $i\neq j$, then algebras $\mathscr{L}^{i}$, $\mathscr{L}^{j}$ are not isomorphic.
\end{corollary}
\begin{proof}	
	If algebras $\mathscr{L}^{i}$, $\mathscr{L}^{j}$, $i\neq j$ are isomorphic, then the algebras $[\mathscr{L}^{i}, \mathscr{L}^{i}]$, $[\mathscr{L}^{j}, \mathscr{L}^{j}]$ are isomorphic as well which results in a contradiction with Theorem~\ref{key999}.
\end{proof}


\begin{thebibliography}{99}

\bibitem{N} Lin Lei, \textquotedblleft Non-alternating Hamiltonian algebra $P(n,m)$ of characteristic two\textquotedblright,  Commun. Algebra \textbf{21}(2), 399--411 (1993).

\bibitem{Kap} I.~Kaplansky, \textquotedblleft Some simple Lie algebras of characteristic 2\textquotedblright, Lecture Notes in Math. \textbf{933}, 127--129 (1982).

\bibitem{bgl} S.~Bouarrouj, P.~Grozman, A.~Lebedev,  D.~Leites, \textquotedblleft Divided power (co)homology. Presentation of simple finite dimensional modular	superalgebras with Cartan matrix\textquotedblright, Homology, Homotopy Appl. \textbf{12}(1), 237--248 (2010).

\bibitem{L} A.~Lebedev, \textquotedblleft Analog of orthogonal, Hamiltonian, and contact Lie	superalgebras in characteristic 2\textquotedblright,  J. Nonlin. Math. Phys. \textbf{17}, Suppl. 1, 399--411 (2010).

\bibitem{dep} M.~I.~Kuznetsov, A.~V.~Kondrateva,  N.~G.~Chebochko, \textquotedblleft Non-alternating Hamiltonian Lie algebras in characteristic 2. I \textquotedblright, (http://arxiv.org/abs/1812.11213). Accessed 2018. 	

\bibitem{KSh} A.~I.~Kostrikin, I.~R.~Shafarevich \textquotedblleft Graded Lie algebras of finite characteristic\textquotedblright, Math. USSR-Izv. \textbf{3}(2), 237--304.	

\bibitem{KKCh}  M.~I.~Kuznetsov, A.~V.~Kondrateva,  N.~G.~Chebochko, \textquotedblleft On Hamiltonian Lie algebras of characteristic 2\textquotedblright, Matem. J. \textbf{16}(2), 54--65 (2016). [in Russian] (available at http://www.math.kz).		

\bibitem{K}  M.~I.~Kuznetsov, \textquotedblleft Truncated induced modules over transitive Lie algebras of characteristic $p$\textquotedblright, Math. USSR-Izv. \textbf{34}, 575--608 (1990).

\bibitem{KK} M.~I.~Kuznetsov, S.~A.~Kirillov, \textquotedblleft Hamiltonian differential forms over the divided powers algebra\textquotedblright, Russian Math. Surveys. \textbf{41}(2), 205--206 (1986).

\bibitem{S} S.~M.~Skryabin, \textquotedblleft Classification of Hamiltonian forms over divided powers\textquotedblright, Math. USSR-Sb. \textbf{69}, 121--141 (1991).

\bibitem{SAr} S.~M.~Skryabin, \textquotedblleft The normal shapes of the symplectic and contact forms over algebras of divided powers\textquotedblright, (https://arxiv.org/abs/1906.11496v1). Accessed 2019.

\bibitem{SS} S.~M.~Skryabin, \textquotedblleft Modular Lie algebras over algebraically non-closed fields. I \textquotedblright, Commun. Algebra. \textbf{19}, 1629--1741 (1991).

\bibitem{Sk} S.~M.~Skryabin, \textquotedblleft An algebraic approach to the Lie algebras of Cartan type\textquotedblright, Commun. Algebra. \textbf{21}, 1229--1336 (1993).

\end{thebibliography}
\end{document}